\documentclass[11pt]{article}
\usepackage[margin=1in]{geometry}
\usepackage{amsmath,amssymb,amsthm,mathtools}
\usepackage[T1]{fontenc}
\usepackage{lmodern,microtype}
\usepackage[colorlinks=true,linkcolor=blue,citecolor=blue,urlcolor=blue]{hyperref}
\hypersetup{pdftitle={Ghirlanda-Guerra identities and cavity equations for diluted spin glass models},pdfauthor={Dmitry Panchenko and Giorgio Parisi}}
\newcommand{\E}{\mathbb E}
\newcommand{\N}{\mathbb N}
\newcommand{\Var}{\operatorname{Var}}
\newcommand{\Pois}{\operatorname{Pois}}
\newcommand{\Av}{\operatorname{Av}}
\newcommand{\TV}{\operatorname{TV}}
\newcommand{\cA}{\mathcal A}
\newcommand{\cP}{\mathcal P}
\newcommand{\cH}{\mathcal H}
\newcommand{\cG}{\mathcal G}
\newcommand{\cZ}{\mathcal Z}
\newcommand{\bx}{\boldsymbol x}
\newcommand{\one}{\mathbf 1}
\newcommand{\la}{\langle}
\newcommand{\ra}{\rangle}
\newtheorem{theorem}{Theorem}[section]
\newtheorem{lemma}[theorem]{Lemma}
\newtheorem{proposition}[theorem]{Proposition}
\newtheorem{corollary}[theorem]{Corollary}
\theoremstyle{remark}

\numberwithin{equation}{section}
\title{Ghirlanda--Guerra identities and cavity equations\\for diluted spin glass models}
\author{Dmitry Panchenko\thanks{\tiny Department of Mathematics, University of Toronto, panchenk@math.toronto.edu. Partially supported by NSERC grant.}
}
\date{}
\begin{document}
\maketitle
\begin{abstract}
This note supplies the details of the cavity lower bound used in
\emph{The M\'ezard-Parisi formula} \cite{MPformula}. Specifically, we
show that small perturbations that do not change the limiting free
energy can be chosen so that the asymptotic Gibbs measures arising
in the Aizenman--Sims--Starr lower bound satisfy both the
Ghirlanda--Guerra identities and the cavity equations. These
calculations are standard and require only minor modifications of
arguments already available in the literature. We record them here
for completeness, avoiding a lengthy repetition of these arguments
in \cite{MPformula}.
\end{abstract}

\section{The model and the statement}

This note supplies the details of the cavity lower bound used in
\emph{The M\'ezard-Parisi formula} \cite{MPformula}. The asymptotic cavity
construction in \cite{Spins} gives limiting spin distributions satisfying
the cavity equations, while the generic Gaussian perturbation in
\cite[Section 3.2]{SKbook} gives the Ghirlanda--Guerra identities. Their
simultaneous use for diluted models was stated in \cite[Theorem 1]{rRSB}.
We give the estimates needed to apply these constructions to a countable
family of unbounded clauses, including the full diluted random energy
model (REM) perturbation used in \cite{MPformula}.

There are two regularizations in addition to the extensive clauses of the
model. A Poisson number of independent cavity factors makes the limiting
spin distribution invariant under the addition of finitely many sites.
An independent generic mixed $p$-spin Gaussian field gives the
Ghirlanda--Guerra identities. Both have vanishing free-energy cost per
spin. The Gaussian covariance changes by $o(1)$ under a fixed cavity
shift, so the two constructions are compatible. We choose the Gaussian
coefficients and the volumes jointly to retain the cavity lower bound.

The general construction is stated in Theorem~\ref{thm:main}; its specialization to
the REM perturbation is Corollary~\ref{cor:A4}, which gives precisely the
input stated in \cite[Theorem 3.1]{MPformula}. In the REM specialization,
all regularization limits are taken at fixed perturbation strength
$\epsilon>0$.

\paragraph{Acknowledgement about the use of AI.}
This note was written by ChatGPT Astra. I had previously carried out these calculations in detail while preparing \cite{rRSB}, but had not written them down.

\subsection{Countably many clause types}

Let $\cA$ be a finite or countable set. Type $a\in\cA$ has arity $d_a\ge2$, intensity $\lambda_a>0$, and a random function
$\theta_a:\{-1,1\}^{d_a}\to\mathbb R$ whose law is invariant under permutations of its inputs. Let $u:\{-1,1\}\to\mathbb R$ be a random unary field, and put
\[
 U=\max_{t=\pm1}|u(t)|,\qquad
 J_a=\max_{\boldsymbol t\in\{-1,1\}^{d_a}}|\theta_a(\boldsymbol t)|.
\]
We assume the moment bounds
\begin{equation}\label{eq:moments}
 \begin{aligned}
 K_1&:=\E U+\sum_{a\in\cA}\lambda_a d_a^2\E J_a<\infty,\\
 K_2&:=\E U^2+\sum_{a\in\cA}\lambda_a d_a^2\E J_a^2<\infty.
 \end{aligned}
\end{equation}
All copies of unary fields, clauses, Poisson variables and indices below are independent, unless a coupling is explicitly specified. For $\Sigma_N=\{-1,1\}^N$ define
\begin{equation}\label{eq:model}
 \begin{aligned}
 H_N(\sigma)&=\sum_{i=1}^N u_i(\sigma_i)
 +\sum_{a\in\cA}\sum_{j=1}^{\pi_a(N\lambda_a)}
       \theta_{a,j}(\sigma_{i_{a,j,1}},\ldots,\sigma_{i_{a,j,d_a}}),
 \\
 F_N&=\frac1N\E\log\sum_{\sigma\in\Sigma_N}e^{H_N(\sigma)}.
 \end{aligned}
\end{equation}
Here $\pi_a(t)$ has Poisson mean $t$, and each index is uniform on $\{1,\ldots,N\}$; repetitions are allowed. The series converges absolutely in supremum norm almost surely, since the expected sum of its absolute clause bounds is $N\sum_a\lambda_a\E J_a<\infty$. Moreover, $(F_N)$ is bounded.

\subsection{The limiting identities and the functional}

A spin law is represented by a random probability measure $G$ on
\[
 \cH=\{m\in L^2([0,1]): |m(v)|\le1\text{ almost everywhere}\}.
\]
Conditionally on $G$, sample independent states $m^\ell\sim G$.
Independently of $G$ and the states, take site variables $v_i$ and thermal
variables $x_{i,\ell}$, all independent and uniform on $[0,1]$. The spins
and off-diagonal overlaps are
\begin{equation}\label{eq:spins}
 \sigma_i^\ell=2\one\{x_{i,\ell}\le(1+m^\ell(v_i))/2\}-1,
 \qquad R_{\ell,\ell'}=\int_0^1m^\ell(v)m^{\ell'}(v)\,dv
 \quad(\ell\ne\ell').
\end{equation}
Brackets denote averaging over states sampled from $G$.
The outer expectation averages the remaining randomness,
including the site variables and thermal variables when they occur in
an observable. In finite volume, brackets denote the Gibbs expectation
over spin configurations, and the outer expectation averages the
disorder and any independently sampled site indices.

The spin Ghirlanda--Guerra identities \cite{GG} are
\begin{equation}\label{eq:GG}
 \E\la f\psi(R_{1,n+1})\ra
 =\frac1n\E\la f\ra\,\E\la\psi(R_{1,2})\ra
    +\frac1n\sum_{\ell=2}^n\E\la f\psi(R_{1,\ell})\ra,
\end{equation}
for $n\ge2$, bounded measurable $f$ of the spins of the first $n$ replicas, and bounded Borel $\psi:[-1,1]\to\mathbb R$.

Extend each clause to magnetizations by its Boltzmann factor:
\begin{equation}\label{eq:extension}
 e^{\theta_a(x_1,\ldots,x_{d_a})}
 =\sum_{\boldsymbol t\in\{-1,1\}^{d_a}}
       e^{\theta_a(\boldsymbol t)}
           \prod_{j=1}^{d_a}\frac{1+t_jx_j}{2}.
\end{equation}
This convention includes inputs fixed at $\pm1$, and $|\theta_a(x)|\le J_a$.
For a function of a cavity spin, write
$\Av_\varsigma h(\varsigma)=(h(1)+h(-1))/2$.
Using a fresh independent site variable for every incoming argument,
define the cavity and correction fields
\begin{align}
 A_i(m,\varsigma)&=u_i(\varsigma)+\sum_{a\in\cA}\sum_{j=1}^{\pi_{i,a}(\lambda_a d_a)}
  \theta_{i,a,j}\bigl(m(v_{i,a,j,1}),\ldots,
                    m(v_{i,a,j,d_a-1}),\varsigma\bigr),\label{eq:A}\\
 B(m)&=\sum_{a\in\cA}\sum_{j=1}^{\widehat\pi_a(\lambda_a(d_a-1))}
  \widehat\theta_{a,j}\bigl(m(\widehat v_{a,j,1}),\ldots,
                          m(\widehat v_{a,j,d_a})\bigr).\label{eq:B}
\end{align}
All this fresh randomness is independent of $G$. Define
\begin{align}
 A_i(m)&=\log\Av_\varsigma e^{A_i(m,\varsigma)},\qquad
 \xi_i(m)=\frac{\Av_\varsigma\varsigma e^{A_i(m,\varsigma)}}{e^{A_i(m)}},\notag\\
 \Lambda(m)&=\exp\sum_{i=1}^s A_i(m),\qquad
 \cZ=\la\Lambda\ra.\label{eq:outputs}
\end{align}
Thus $e^{A_i(m)}$ is the cavity multiplier and $\xi_i(m)$ is the cavity
output. The number $s$ of cavity sites in $\Lambda$ and $\cZ$ is always
fixed in the identity under consideration.
For integers $s,n\ge1$, $L_0\ge s$, and sets
$C_\ell\subseteq\{1,\ldots,L_0\}$, $1\le\ell\le n$, let
$C_\ell^1=C_\ell\cap\{1,\ldots,s\}$ and $C_\ell^2=C_\ell\setminus C_\ell^1$.
The cavity equations are
\begin{equation}\label{eq:SC}
 \E\prod_{\ell=1}^n\left\la\prod_{i\in C_\ell}m(v_i)\right\ra
 =\E\prod_{\ell=1}^n
  \frac{\left\la\Lambda(m)
          \prod_{i\in C_\ell^1}\xi_i(m)
          \prod_{i\in C_\ell^2}m(v_i)\right\ra}
       {\cZ}.
\end{equation}
The noncavity coordinates on the right are independent of the incoming coordinates in \eqref{eq:A}. These are the equations with no extra borrowed clauses in \cite[Theorem 1]{Spins}, in the magnetization notation of \cite{rRSB}.

The Aizenman--Sims--Starr cavity functional \cite{ASS,Spins} is
\begin{equation}\label{eq:functional}
 \cP(G)=\log2+\E\log\la e^{A_1}\ra-\E\log\la e^B\ra.
\end{equation}
All series and expectations in \eqref{eq:A}--\eqref{eq:functional} are
well defined under \eqref{eq:moments}. The fields $A_i(m,\varsigma)$,
$A_i(m)$ and $B(m)$ are bounded uniformly in their spin and state
arguments by sums of absolute clause bounds with finite expectation.
The same envelopes bound the logarithms in \eqref{eq:functional}.
Moreover, $|\xi_i(m)|\le1$ and $\cZ>0$ almost surely.

\begin{theorem}[Joint regularization and cavity lower bound]\label{thm:main}
Under \eqref{eq:moments}, there is a random spin law $G$ satisfying both \eqref{eq:GG} and \eqref{eq:SC} such that
\begin{equation}\label{eq:mainbound}
 \liminf_{N\to\infty}F_N\ge\cP(G).
\end{equation}
More precisely, $G$ is a subsequential spin limit of the regularized models defined below, with deterministic choices of the Gaussian coefficients. Their free energy per spin differs from $F_N$ by $o(1)$ uniformly over those coefficients. Along the selected sequence, their one-site cavity increments converge to $\cP(G)$.
\end{theorem}

The lower bound does not use a sign or convexity condition on the clauses. It applies, in particular, to models for which a matching interpolation upper bound is unavailable. The proof occupies Sections~\ref{sec:regularizations}--\ref{sec:selection}. Section~\ref{sec:REM} verifies the assumptions and the exact fields for the REM perturbation.

\section{The two regularizations}\label{sec:regularizations}

We first specify the regularized finite models. Their two additional terms are independent of one another and of $H_N$.

For each $\ell$, let $A_{N,\ell}(\sigma,\varsigma)$ be a fresh copy of the finite-volume cavity field
\begin{equation}\label{eq:AN}
 A_{N,\ell}(\sigma,\varsigma)=u_\ell(\varsigma)
  +\sum_{a\in\cA}\sum_{j=1}^{\pi_{\ell,a}(\lambda_a d_a)}
       \theta_{\ell,a,j}
          (\sigma_{i_{\ell,a,j,1}},\ldots,
           \sigma_{i_{\ell,a,j,d_a-1}},\varsigma),
\end{equation}
with incoming indices uniform on $\{1,\ldots,N\}$. Set
\begin{equation}\label{eq:reservoir}
 C_N(\sigma)=\sum_{\ell=1}^{\pi(c_N)}
                     \log\Av_\varsigma e^{A_{N,\ell}(\sigma,\varsigma)},
 \qquad c_N=N^{1/4}.
\end{equation}
This is the cavity regularization of \cite{Spins}, with the full field of the present model.

For $p\ge1$, let $g_{N,p}$ be independent centered Gaussian processes with
\begin{equation}\label{eq:g}
 g_{N,p}(\sigma)=N^{-p/2}\sum_{i_1,\ldots,i_p=1}^N
                       g_{i_1\ldots i_p}^{(p)}\sigma_{i_1}\cdots\sigma_{i_p},
 \qquad \E g_{N,p}(\sigma)g_{N,p}(\tau)=R_N(\sigma,\tau)^p,
\end{equation}
Here all coefficients $g_{i_1\ldots i_p}^{(p)}$ are independent standard
Gaussians, and $R_N(\sigma,\tau)=N^{-1}\sum_i\sigma_i\tau_i$.
For a coefficient sequence $\bx=(x_p)_{p\ge1}$, put
\begin{equation}\label{eq:generic}
 g_{N,\bx}=\sum_{p\ge1}2^{-p}x_pg_{N,p},\qquad
 Q_{\bx}(t)=\sum_{p\ge1}4^{-p}x_p^2t^p,\qquad s_N=N^{3/8}.
\end{equation}
We use $\bx\in[1,2]^{\N}$, and allow $\bx\in[0,3]^{\N}$ in the concentration estimates. For each fixed $N$ the series converges almost surely uniformly over $\Sigma_N$ and over these coefficient sequences. Define
\begin{equation}\label{eq:regularized}
 \widehat H_{N,\bx}=H_N+C_N+s_Ng_{N,\bx},\qquad
 \widehat Z_N(\bx)=\sum_{\sigma\in\Sigma_N}e^{\widehat H_{N,\bx}(\sigma)}.
\end{equation}
Let $\la\cdot\ra_{N,\bx}$ be its Gibbs expectation, and let
$\mu_{N,\bx}$ denote the annealed law of its replica spin array, using
the sampled-site convention of Section~\ref{sec:spincompactness}.
The coefficients $\bx$ are deterministic when the disorder is averaged.

An absolute bound for one cavity factor has the law
\begin{equation}\label{eq:envelope}
 W=U+\sum_{a\in\cA}\sum_{j=1}^{\pi_a(\lambda_a d_a)}J_{a,j}.
\end{equation}
In particular, $\E W^2<\infty$. This follows from the compound-Poisson second-moment formula and \eqref{eq:moments}; the square of the mean and the sum $\sum_a\lambda_a d_a\E J_a^2$ are both finite.

\begin{lemma}[Free-energy cost and concentration]\label{lem:pressureconc}
Uniformly over $\bx\in[0,3]^{\N}$,
\begin{align}
 \left|N^{-1}\E\log\widehat Z_N(\bx)-F_N\right|
      &\le\frac{c_N\E W}{N}+\frac{3s_N^2}{2N},\label{eq:cost}\\
 \Var(\log\widehat Z_N(\bx))
      &\le 2N\E U^2+N\sum_a\lambda_a\E J_a^2
                      +c_N\E W^2+3s_N^2.\label{eq:variance}
\end{align}
Consequently
\begin{equation}\label{eq:vN}
 \sup_{\bx}\E\left|\log\widehat Z_N(\bx)
                         -\E\log\widehat Z_N(\bx)\right|
             \le v_N:=C\sqrt N,
 \qquad v_N/s_N^2\longrightarrow0.
\end{equation}
\end{lemma}
\begin{proof}
Changing a Hamiltonian by a function of supremum norm $b$ changes its log partition function by at most $b$. Thus adding $C_N$ costs at most $c_N\E W$ in expected log partition function. Conditional on the other disorder, Jensen's inequality bounds the expected Gaussian contribution between zero and $s_N^2Q_{\bx}(1)/2$. Since $Q_{\bx}(1)\le3$, this proves \eqref{eq:cost}.

For a marked Poisson process, the variance of a square-integrable functional is bounded by the integral of the squared change produced by adding one marked point. The inequality follows, for a finite intensity, by tensorizing the Efron--Stein inequality for independent Bernoulli approximations of the process and passing to the Poisson limit. Apply it to the model clauses and to the cavity factors: their add-one bounds are respectively $J_a$ and $W$. Resampling a unary field changes the log partition function by at most $U+U'$, so its Efron--Stein contribution is at most $2\E U^2$. The Gaussian Poincar\'e inequality gives $s_N^2Q_{\bx}(1)$ for the last term. Tensorization over these independent sources proves \eqref{eq:variance}. Truncation to finitely many types and Gaussian components, followed by $L^2$ convergence, justifies the same calculation for the full series. Cauchy--Schwarz now gives \eqref{eq:vN}.
\end{proof}

\section{The spin Ghirlanda--Guerra estimates}\label{sec:GG}

We record the Gaussian argument with a bound independent of the spin test. Let $\nu$ be product uniform measure on $[1,2]^{\N}$ and define
\begin{equation}\label{eq:concentrationdefect}
 e_{N,p}(\bx)=\frac{1}{s_N2^{-p}x_p}
   \E\la|g_{N,p}(\sigma)-\E\la g_{N,p}(\sigma)\ra_{N,\bx}|\ra_{N,\bx}.
\end{equation}
For each fixed $p$,
\begin{equation}\label{eq:epaverage}
 \int e_{N,p}(\bx)\,d\nu(\bx)\longrightarrow0.
\end{equation}

Here is the concentration calculation, as in \cite[Section 3.2]{SKbook}. Fix all coefficients except $x_p=t$, set $b=s_N2^{-p}$, and write
$\varphi(t)=\log\widehat Z_N(\bx)$ and $\phi(t)=\E\varphi(t)$. Differentiation and Gaussian integration by parts give
\begin{equation}\label{eq:convexderivatives}
 \varphi'(t)=b\la g_{N,p}\ra,\qquad
 \phi'(t)=tb^2(1-\E\la R_{1,2}^p\ra),\qquad
 \phi''(t)=b^2\E\la(g_{N,p}-\la g_{N,p}\ra)^2\ra.
\end{equation}
Thus $0\le\phi'(t)\le2tb^2$ on $[0,3]$, and integration over $[1,2]$ gives
\[
 \int_1^2\E\la|g_{N,p}-\la g_{N,p}\ra|\ra\,dt\le2.
\]
For differentiable convex functions, comparison with difference quotients at $t-y,t,t+y$ gives
\[
 \E|\varphi'(t)-\phi'(t)|
 \le\phi'(t+y)-\phi'(t-y)+3v_N/y\qquad(0<y\le1).
\]
Integrating this inequality and using the bound on $\phi'$ yields
\[
 \int_1^2\E|\la g_{N,p}\ra-\E\la g_{N,p}\ra|\,dt
       \le24\left(yb+\frac{v_N}{yb}\right).
\]
For fixed $p$ and large $N$, the choice $y=\sqrt{v_N}/b$ is at most one. Consequently,
\begin{equation}\label{eq:concentrationbound}
 \int e_{N,p}(\bx)\,d\nu(\bx)
       \le\frac{2^p}{s_N}\left(2+48\sqrt{v_N}\right)
       \longrightarrow0.
\end{equation}
The convergence follows from $s_N=N^{3/8}$ and $v_N\le C\sqrt N$.

For $|f|\le1$ depending on the spins of the first $n$ replicas, Gaussian integration by parts gives
\begin{align}
 &\left|\E\la fR_{1,n+1}^p\ra_{N,\bx}
 -\frac1n\E\la f\ra_{N,\bx}\E\la R_{1,2}^p\ra_{N,\bx}
 -\frac1n\sum_{\ell=2}^n\E\la fR_{1,\ell}^p\ra_{N,\bx}\right|
 \le\frac{e_{N,p}(\bx)}n.\label{eq:GGfinite}
\end{align}
Indeed, subtract $\E\la f\ra\E\la g_{N,p}\ra$ from $\E\la fg_{N,p}(\sigma^1)\ra$. The Gaussian formula identifies this difference with $-ns_N2^{-p}x_p$ times the signed expression inside \eqref{eq:GGfinite}. Its absolute value is bounded by the numerator in \eqref{eq:concentrationdefect}. This proof permits every bounded spin test $f$; no restriction to overlap tests is made.

Define one measurable defect controlling all $p$ by
\begin{equation}\label{eq:Gamma}
 \Gamma_N(\bx)=\sum_{p\ge1}2^{-p}\min\{1,e_{N,p}(\bx)\},
 \qquad r_N=\int\Gamma_N\,d\nu\longrightarrow0.
\end{equation}
The last assertion follows from \eqref{eq:epaverage} and dominated convergence. A deterministic sequence with $\Gamma_N(\bx^{(N)})\to0$ therefore satisfies every finite-volume identity \eqref{eq:GGfinite} asymptotically. We will choose such coefficients together with the lower-bound sequence in Section~\ref{sec:selection}.

\section{Compatibility with finite cavity shifts}\label{sec:shifts}

All comparisons in this section use the same coefficient sequence $\bx$ at volumes $N$ and $N+s$, with $s$ fixed. The error bounds are uniform over $\bx\in[0,3]^{\N}$.

\subsection{Two comparison facts}

If two Hamiltonians on the same spin space satisfy $\|H-H'\|_\infty\le D$, then, for a bounded observable $f$ of $n$ replicas,
\begin{equation}\label{eq:Gibbscompare}
 |\log Z_H-\log Z_{H'}|\le D,\qquad
 |\la f\ra_H-\la f\ra_{H'}|\le2n\|f\|_\infty D.
\end{equation}
The second inequality follows by interpolation and differentiation of the product Gibbs measure. Both inequalities can be averaged when $D$ is random.

If instead the two Hamiltonians differ in an independent centered Gaussian field, and the two covariance kernels differ uniformly by at most $b$, Gaussian interpolation gives
\begin{equation}\label{eq:Gaussiancompare}
 |\E\log Z_1-\E\log Z_0|\le b,\qquad
 |\E\la f\ra_1-\E\la f\ra_0|\le C_n\|f\|_\infty b.
\end{equation}
To see this, interpolate by $\sqrt t\,g_1+\sqrt{1-t}\,g_0$. Gaussian integration by parts expresses the derivative of the expected log partition function as one half the difference of a diagonal and a two-replica covariance. For an $n$-replica observable it gives a finite sum of covariance differences involving at most $n+2$ replicas, with total absolute coefficient bounded in terms of $n$. Integration proves \eqref{eq:Gaussiancompare} for any independent background Hamiltonian.

\begin{lemma}[Poisson count changes]\label{lem:Poisson}
Adding a fixed number of independent cavity factors to $C_N$ changes
the annealed expectation of any bounded spin observable of a fixed
number of replicas by $o(1)$. Also, replacing the model count means $N\lambda_a$ by
\begin{equation}\label{eq:bulkmeans}
 \beta_{N,s,a}=\lambda_a(N+s)\left(\frac{N}{N+s}\right)^{d_a}
             =\lambda_a N(1+s/N)^{1-d_a}
\end{equation}
changes every such annealed spin expectation by $o(1)$, uniformly over
any independent background, including the two regularizations.
\end{lemma}
\begin{proof}
Write $\TV$ for total variation distance. For $K\sim\Pois(t)$, direct
comparison of its mass function with that of $K+1$ gives
\[
 \TV(\mathcal L(K+1),\mathcal L(K))
       =\tfrac12\E|K/t-1|\le(2\sqrt t)^{-1}.
\]
By the triangle inequality, the bound for a fixed shift $s$ is $s/(2\sqrt t)$. Apply this with $t=c_N$; conditionally on the count, the cavity factors have the same independent mark law.

For the second assertion, write
\begin{equation}\label{eq:deltameans}
 \delta_{N,s,a}=N\lambda_a-\beta_{N,s,a},\qquad
 0\le\delta_{N,s,a}\le s\lambda_a(d_a-1).
\end{equation}
For any fixed finite set of types, the two Poisson count laws approach one another in total variation. For example, couple the larger count as the smaller plus an independent Poisson variable of bounded mean, and average the preceding shift bound; $\beta_{N,s,a}\to\infty$ for each fixed $a$.

For the other types, couple the counts by adding independent clauses with means $\delta_{N,s,a}$. Their expected total absolute bound is at most
\[
 s\sum_{a\text{ outside the finite set}}\lambda_a(d_a-1)\E J_a.
\]
By \eqref{eq:Gibbscompare}, this bounds the change of any fixed bounded replica test up to its factor $2n\|f\|_\infty$. First let $N$ tend to infinity for the finite set, and then exhaust $\cA$. The tail tends to zero by \eqref{eq:moments}. This argument applies without a finite total intensity $\sum_a\lambda_a$.
\end{proof}

\subsection{The Gaussian field under a finite shift}

Write a configuration on $N+s$ sites as $(\rho,\varsigma)\in\Sigma_N\times\{-1,1\}^s$. For two such configurations,
\[
 R_{N+s}=\frac{N}{N+s}R_N(\rho,\rho')
          +\frac1{N+s}\sum_{i=1}^s\varsigma_i\varsigma_i',
 \qquad |R_{N+s}-R_N|\le\frac{2s}{N+s}.
\]
Since $\sup_{\bx}(\|Q_{\bx}\|_\infty+\|Q_{\bx}'\|_\infty)<\infty$ on $[-1,1]$,
\begin{equation}\label{eq:Gaussianerror}
 \sup_{\rho,\rho',\varsigma,\varsigma'}
 \left|s_{N+s}^2Q_{\bx}(R_{N+s})-s_N^2Q_{\bx}(R_N)\right|
 \le C\left(|s_{N+s}^2-s_N^2|+\frac{s s_{N+s}^2}{N}\right)
 =O_s(N^{-1/4}).
\end{equation}
Thus \eqref{eq:Gaussiancompare} replaces the Gaussian field on $N+s$ sites by the field on the $N$ bulk sites, constant in the cavity spins, at a cost tending to zero both for expected log partition functions and for fixed replica observables.

\subsection{The cavity factors under a finite shift}

In each incoming index of a cavity factor on $N+s$ sites, keep the index if it belongs to the bulk and otherwise replace it by a fresh uniform bulk index. The resulting indices are independent uniform indices on $N$ sites. The probability that a type-$a$ clause in such a factor is affected is at most $s(d_a-1)/(N+s)$. Since changing its arguments changes the clause by at most $2J_a$, the coupled cavity regularizations differ in expected supremum norm by at most
\begin{equation}\label{eq:reservoirerror}
 \frac{2s c_{N+s}}{N+s}\sum_a\lambda_a d_a(d_a-1)\E J_a
                  +|c_{N+s}-c_N|\E W=o(1).
\end{equation}
The last term couples the Poisson counts to mean $c_N$. This estimate uses the total size of the affected clauses and remains valid when there are infinitely many incoming indices.

\subsection{Splitting the extensive clauses}

Split each type-$a$ Poisson process on $N+s$ sites according to the number of its indices in the cavity set. The bulk terms have mean \eqref{eq:bulkmeans}. For a specified cavity site, the terms with exactly one occurrence of that site and all other indices in the bulk have mean
\begin{equation}\label{eq:singlemeans}
 \gamma_{N,s,a}=\lambda_a d_a\left(\frac N{N+s}\right)^{d_a-1}.
\end{equation}
Different groups are independent. Input-permutation symmetry gives the field in \eqref{eq:AN}, after placing the cavity spin in its last argument.

The expected absolute bound of all terms with at least two cavity indices is at most
\begin{equation}\label{eq:multihit}
 \frac{s^2}{2(N+s)}\sum_a\lambda_a d_a(d_a-1)\E J_a.
\end{equation}
Indeed, the probability of at least two hits in $d_a$ independent draws is bounded by $\binom{d_a}{2}(s/(N+s))^2$. The incoming counts in \eqref{eq:singlemeans} can be increased to their target means $\lambda_a d_a$ with expected total clause bound, over all $s$ cavity sites, at most
\begin{equation}\label{eq:singleerror}
 \frac{s^2}{N+s}\sum_a\lambda_a d_a(d_a-1)\E J_a.
\end{equation}
Both errors are $o(1)$ by \eqref{eq:moments}. The unary fields split exactly into $N$ bulk fields and $s$ fresh cavity fields.

\begin{proposition}[Stability under a fixed volume shift]\label{prop:shift}
For each bounded observable $f$ of finitely many sites and replicas, and each fixed $s\ge1$,
\begin{equation}\label{eq:shift}
 \sup_{\bx\in[0,3]^{\N}}
       |\E\la f\ra_{N+s,\bx}-\E\la f\ra_{N,\bx}|\longrightarrow0.
\end{equation}
The sites in $f$ are regarded as bulk sites in this comparison.
\end{proposition}
\begin{proof}
Apply \eqref{eq:Gaussianerror}, \eqref{eq:reservoirerror}, \eqref{eq:multihit} and \eqref{eq:singleerror}. Up to an error tending to zero, the Hamiltonian on $N+s$ sites becomes
\begin{equation}\label{eq:split}
 H_{N,s}^{\rm bulk}(\rho)+C_N(\rho)+s_Ng_{N,\bx}(\rho)
                          +\sum_{i=1}^sA_{N,i}(\rho,\varsigma_i),
\end{equation}
where the bulk Hamiltonian has model count means $\beta_{N,s,a}$ and its original $N$ unary fields. All $s$ displayed cavity fields are fresh and independent of the bulk and of its regularizations.

Integrating the cavity spins in a bulk observable adds the $s$ factors
$\log\Av_\varsigma e^{A_{N,i}(\rho,\varsigma)}$ to $C_N$. They have exactly the same mark law as its original terms. Lemma~\ref{lem:Poisson} absorbs this fixed count shift and replaces the bulk means by $N\lambda_a$. This gives \eqref{eq:shift}. Every error used above is uniform in $\bx$.
\end{proof}

\section{Spin limits and the cavity equations}\label{sec:limits}

We next explain the passage to limiting spin laws. This also establishes the continuity needed for the cavity functional.

\subsection{Spin compactness and overlaps}\label{sec:spincompactness}

For each $N$, sample replicas from the finite Gibbs measure and,
independently, an infinite sequence of uniform site indices
$I_i\in\{1,\ldots,N\}$, shared across replicas. Record
$\sigma_{I_i}^\ell$. Its annealed law is $\mu_{N,\bx}$. This infinite
array is separately exchangeable in its site and replica indices.
For every fixed number of sampled sites, site symmetry implies that
its law differs from the law at distinct prescribed sites by at most
the probability of an index collision, which tends to zero.
Compactness of $\{-1,1\}^{\N\times\N}$ gives a convergent subsequence
of these spin laws. The Aldous--Hoover representation
\cite{Aldous,Hoover}, followed by averaging the cell-specific variable,
gives precisely \eqref{eq:spins} for a random $G$.

The off-diagonal overlaps converge jointly with every finite spin array. To verify this, condition on all finite-volume replicas. For distinct $\ell,\ell'$, independent sampled indices give
\begin{equation}\label{eq:empiricaloverlap}
 \E\left[\left|
  \frac1M\sum_{i=1}^M\sigma_{I_i}^\ell\sigma_{I_i}^{\ell'}
          -R_N(\sigma^\ell,\sigma^{\ell'})\right|^2
            \biggm| (\sigma^j)_j\right]\le\frac1M.
\end{equation}
Use additional independent site samples beyond those in the finite
test. Pass first to the spin-array limit for fixed $M$, then let
$M\to\infty$. In the representation \eqref{eq:spins}, the empirical
product tends almost surely to $\int m^\ell m^{\ell'}$.
This proves joint convergence with the off-diagonal overlaps.
The diagonal finite-volume overlap is one, whereas
$\int m^\ell(v)^2\,dv$ need not be one; only off-diagonal overlaps
enter the identities considered here.

\subsection{Fresh-clause observables}

\begin{lemma}[Continuity of fresh-clause observables]\label{lem:continuity}
Suppose annealed spin laws on $N$ sites converge to a law represented by $G$. For fresh cavity fields \eqref{eq:AN} and fresh correction clauses with means $\lambda_a(d_a-1)$, the following converge to their counterparts defined by \eqref{eq:A}--\eqref{eq:outputs}: the two expected log partition ratios in \eqref{eq:functional}, and every bounded replicated cavity-spin moment on the right of \eqref{eq:SC}.
\end{lemma}
\begin{proof}
First retain finitely many types, condition on their Poisson counts and clause tables, and restrict these counts and absolute bounds to bounded sets. Only finitely many fresh incoming indices are then used. Their collisions with one another or with the prescribed test sites have probability tending to zero. On the complement, all finite Gibbs expectations involved concern a fixed finite set of sites.

Convergence of the replicated spin array gives convergence of all joint moments of these Gibbs expectations. They take values in a compact cube. Polynomial approximation therefore gives convergence against continuous functions of them, including the relevant log partition functions and ratios: on the bounded disorder event their positive denominators are bounded away from zero. In the limit, averaging the thermal coins at distinct incoming sites gives \eqref{eq:extension}, and the resulting fields are exactly \eqref{eq:A} and \eqref{eq:B}.

Remove the bounds on the finitely many counts and clause tables. The replicated spin ratios have absolute value at most one. The log ratios are bounded by the sums of the absolute clause bounds, whose expectations are finite; the distribution of these envelopes is independent of $N$. This gives the required uniform integrability.

Finally exhaust the clause types. Removing a collection $\mathcal T$ of
types changes the site log ratio in absolute expectation by at most
$\sum_{a\in\mathcal T}\lambda_a d_a\E J_a$, and the correction log
ratio by at most $\sum_{a\in\mathcal T}\lambda_a(d_a-1)\E J_a$.
For a spin moment of $n$ replicas at $s$ cavity sites, the bound is
\begin{equation}\label{eq:cavitytail}
 2ns\sum_{a\in\mathcal T}\lambda_a d_a\E J_a.
\end{equation}
To obtain the last bound, keep the incoming thermal spins and the cavity spins as auxiliary variables before taking their averages, and apply \eqref{eq:Gibbscompare} to the sum of the removed clauses. The same argument holds in the limiting representation. These tails tend to zero by \eqref{eq:moments}, completing the proof.
\end{proof}

\begin{proposition}[Cavity equations in the limit]\label{prop:cavity}
Every subsequential spin limit of $\mu_{N_j,\bx^{(j)}}$, for arbitrary deterministic $\bx^{(j)}\in[1,2]^{\N}$, satisfies all the cavity equations \eqref{eq:SC}.
\end{proposition}
\begin{proof}
Fix $s,L_0,n$ and the sets $C_\ell$. By Proposition~\ref{prop:shift}, the laws at volumes $N_j+s$ with the same coefficients $\bx^{(j)}$ have the same spin limit. By site exchangeability, designate the first $s$ sites occurring in the moment as cavity sites and apply the decomposition \eqref{eq:split}. Before averaging the disorder, the replicated moment in this decomposed system is
\begin{equation}\label{eq:finitecavityratio}
 \prod_{\ell=1}^n
 \frac{\left\la
    \Av_{\boldsymbol\varsigma}
      \prod_{i\in C_\ell^1}\varsigma_i
      \prod_{i\in C_\ell^2}\rho_i
      \exp\sum_{i=1}^s A_{N_j,i}(\rho,\varsigma_i)
                      \right\ra^{\rm bulk}}
      {\left\la\Av_{\boldsymbol\varsigma}
      \exp\sum_{i=1}^s A_{N_j,i}(\rho,\varsigma_i)
                      \right\ra^{\rm bulk}}.
\end{equation}
The noncavity indices have been relabeled inside the bulk. The background Gibbs measure here contains $C_{N_j}+s_{N_j}g_{N_j,\bx^{(j)}}$ and model means $\beta_{N_j,s,a}$. Its spin law has the same limit $G$ by Lemma~\ref{lem:Poisson}. The fresh cavity fields are independent of this background. Lemma~\ref{lem:continuity} identifies the limit of the averaged expression \eqref{eq:finitecavityratio} with the right side of \eqref{eq:SC}. Its left-side finite-volume moment converges to the left side of that equation. All preceding decomposition errors vanish uniformly in $\bx^{(j)}$.
\end{proof}

The same cavity identity holds for a joint spin moment multiplied by a
bounded Borel function of the off-diagonal overlaps, with that function
inserted inside the replicated Gibbs expectation on each side.
To obtain this extension, add fresh noncavity sites to represent
polynomials in the overlaps, pass to the empirical-overlap limit, and
then use polynomial approximation for continuous tests and equality of
finite signed measures for Borel tests. The overlaps on the transformed
side are those of the original states sampled under the cavity
reweighting.

\begin{proposition}[Spin Ghirlanda--Guerra identities in the limit]\label{prop:GGlimit}
If $\mu_{N_j,\bx^{(j)}}$ converges to a law represented by $G$ and
$\Gamma_{N_j}(\bx^{(j)})\to0$, then $G$ satisfies \eqref{eq:GG}.
\end{proposition}
\begin{proof}
For each fixed $p$, \eqref{eq:Gamma} implies $e_{N_j,p}(\bx^{(j)})\to0$. Apply \eqref{eq:GGfinite} to a cylinder function of the spin array and use the joint convergence with overlaps proved from \eqref{eq:empiricaloverlap}. This gives \eqref{eq:GG} for cylinder $f$ and $\psi(t)=t^p$. Constants and uniform polynomial approximation give every continuous $\psi$. Equality of finite signed measures then gives every bounded Borel $\psi$. The monotone-class theorem extends the spin test from cylinder functions to all bounded functions measurable with respect to the first $n$ spin columns.
\end{proof}

\section{The cavity increment and simultaneous selection}\label{sec:selection}

The remaining issue is to choose the coefficients satisfying the Gaussian estimates at a volume where the cavity increment gives a lower bound. Throughout an individual comparison the coefficients are frozen.

\subsection{The one-site increment}

Set
\begin{equation}\label{eq:increment}
 D_N(\bx)=\E\log\widehat Z_{N+1}(\bx)-\E\log\widehat Z_N(\bx).
\end{equation}
Let $\la\cdot\ra_{N,\bx}'$ denote the Gibbs expectation on the $N$ bulk
sites with model means $\beta_{N,1,a}$, the original unary fields, and
regularizations $C_N+s_Ng_{N,\bx}$. Let $B_N(\rho)$ be a fresh sum of
type-$a$ clauses of means $\delta_{N,1,a}$, with independent uniform
indices in the bulk. Let $A_N$ be a fresh field \eqref{eq:AN} independent
of $B_N$ and of the bulk disorder. The decomposition of
Section~\ref{sec:shifts} gives
\begin{equation}\label{eq:incrementrepresentation}
 D_N(\bx)=\log2
       +\E\log\la\Av_\varsigma e^{A_N(\rho,\varsigma)}\ra_{N,\bx}'
       -\E\log\la e^{B_N(\rho)}\ra_{N,\bx}'+o(1),
\end{equation}
where the error is uniform in $\bx\in[0,3]^{\N}$.

For clarity, the common bulk partition function is used on both sides. Adding the cavity site and its incident clauses gives the numerator partition function $\widehat Z_{N+1}$, up to the errors \eqref{eq:Gaussianerror}, \eqref{eq:reservoirerror}, \eqref{eq:multihit} and \eqref{eq:singleerror}. Adding $B_N$ instead restores the bulk clause means exactly to $N\lambda_a$, and hence gives $\widehat Z_N$. There is no correction unary field.

The correction means converge in the weighted norm needed here. The elementary inequality
\begin{equation}\label{eq:deltalimit}
 0\le\lambda_a(d_a-1)-\delta_{N,1,a}
        \le\frac{\lambda_a d_a(d_a-1)}{2N}
\end{equation}
follows by Taylor's theorem for $(1+t)^{-(d_a-1)}$ on $[0,1/N]$. Thus $B_N$ can be coupled to the target correction field with expected supremum error at most $K_1/(2N)$.

The two log ratios in \eqref{eq:incrementrepresentation} are bounded in absolute value by the corresponding fresh clause envelopes. It follows, enlarging a constant to cover finitely many small $N$, that
\begin{equation}\label{eq:boundedD}
 \sup_{N\ge1}\sup_{\bx\in[1,2]^{\N}}|D_N(\bx)|\le C<\infty.
\end{equation}
Moreover, whenever $\mu_{N_j,\bx^{(j)}}$ converges to $G$, Lemma~\ref{lem:Poisson} gives the same spin limit for its common bulk measure. By Lemma~\ref{lem:continuity}, \eqref{eq:deltalimit} and \eqref{eq:incrementrepresentation},
\begin{equation}\label{eq:incrementlimit}
 D_{N_j}(\bx^{(j)})\longrightarrow\cP(G).
\end{equation}

\subsection{Choosing the volume and the coefficients}

Let $a_N=\int\E\log\widehat Z_N(\bx)\,d\nu(\bx)$ and
$\overline D_N=\int D_N(\bx)\,d\nu(\bx)=a_{N+1}-a_N$. By \eqref{eq:cost},
\begin{equation}\label{eq:Cesaro}
 \liminf_N\overline D_N\le\liminf_N\frac{a_N}{N}
                  =\liminf_NF_N=:L.
\end{equation}
The first inequality is the elementary inequality between the lower limit of a sequence and the lower limit of its Ces\`aro averages, applied to the increments. Choose $N_j\to\infty$ with $\overline D_{N_j}\le L+o(1)$.

Put $\eta_N=\sqrt{r_N}+N^{-1}$ and
\[
 \mathcal X_N=\{\bx:\Gamma_N(\bx)\le\eta_N\},\qquad
 b_N=\nu(\mathcal X_N^c)\le r_N/\eta_N\longrightarrow0.
\]
The uniform bound \eqref{eq:boundedD} implies
\begin{equation}\label{eq:selection}
 \inf_{\bx\in\mathcal X_N}D_N(\bx)
  \le\frac{\overline D_N+Cb_N}{1-b_N}
  \le\overline D_N+\frac{2Cb_N}{1-b_N}
\end{equation}
for all sufficiently large $N$. Choose deterministic $\bx^{(j)}\in\mathcal X_{N_j}$ whose increment is within $1/j$ of this infimum. Then
\begin{equation}\label{eq:selected}
 \Gamma_{N_j}(\bx^{(j)})\longrightarrow0,
 \qquad \limsup_jD_{N_j}(\bx^{(j)})\le L.
\end{equation}

Extract a subsequence along which the spin laws converge to $G$. Proposition~\ref{prop:cavity} gives the cavity equations, Proposition~\ref{prop:GGlimit} gives the spin Ghirlanda--Guerra identities, and \eqref{eq:incrementlimit} together with \eqref{eq:selected} gives $\cP(G)\le L$. This proves Theorem~\ref{thm:main}.

The averaging over $\bx$ was used only to find a good deterministic coefficient sequence. The limiting law itself is obtained with those coefficients fixed before averaging the disorder. In the shift and increment comparisons at $N_j$, the same $\bx^{(j)}$ is used at $N_j+s$; the proof makes no regularity assumption on the relation between $\bx^{(j)}$ and $\bx^{(j+1)}$.

\section{The diluted REM perturbation}\label{sec:REM}

We now specialize Theorem~\ref{thm:main} to the model in
\cite[Section 3]{MPformula}. Take an original clause
$\theta:\{-1,1\}^K\to\mathbb R$, with fixed arity $K\ge2$,
input-permutation-invariant law, and intensity $\lambda>0$. Set
$J_0=\max_{\boldsymbol t\in\{-1,1\}^K}|\theta(\boldsymbol t)|$ and
assume $\E e^{tJ_0}<\infty$ for every $t>0$.
Only $\E J_0^2<\infty$ is needed for Theorem~\ref{thm:main}; the
exponential assumption also supplies the integrability used in the
reproducibility argument. Fix $\epsilon>0$ and set
\begin{equation}\label{eq:REM}
 u_i(\varsigma)=g_i\frac{1+\varsigma}{2},\qquad g_i\sim N(0,\epsilon),\qquad
 \theta_d^\epsilon(\boldsymbol t)\sim
        N(0,\epsilon2^{-4d})\quad
        (\boldsymbol t\in\{-1,1\}^d,\ d\ge2).
\end{equation}
For each $d$, the $2^d$ table entries are independent. All tables, unary
fields, Poisson counts, and sampled indices are mutually independent
and independent of the original model. Add $\Pois(N)$ REM clauses of
each arity $d\ge2$ to the original $\Pois(\lambda N)$ clauses. Thus,
\begin{align}
 H_N^{\mathrm{pert},\epsilon}(\sigma)
 &=\sum_{i=1}^Ng_i\frac{1+\sigma_i}{2}
   +\sum_{d\ge2}\sum_{j=1}^{\pi_d(N)}
        \theta_{d,j}^\epsilon(\sigma_{i_{d,j,1}},\ldots,
                                      \sigma_{i_{d,j,d}}),\notag\\
 H_N^\epsilon&=H_N^{\rm model}+H_N^{\mathrm{pert},\epsilon},\qquad
 F_N^\epsilon=\frac1N\E\log\sum_{\sigma\in\Sigma_N}
                                      e^{H_N^\epsilon(\sigma)}.
                                      \label{eq:REMHamiltonian}
\end{align}
The original clauses and the REM clauses of arity $K$ are separate
types in the general construction, even though their arities coincide.

For $J_d=\max|\theta_d^\epsilon|$, the Gaussian maximum bound gives
\begin{equation}\label{eq:REMmoments}
 \E J_d\le C\sqrt\epsilon\sqrt d\,2^{-2d},\qquad
 \E J_d^2\le C\epsilon d\,2^{-4d}.
\end{equation}
For example, these follow by integrating the union bound for the tails of $2^d$ absolute Gaussian variables. Consequently,
\[
 \sum_{d\ge2}d^2\E J_d<\infty,
 \qquad \sum_{d\ge2}d^2\E J_d^2<\infty,
\]
so \eqref{eq:moments} holds for the full model \eqref{eq:REMHamiltonian}. Theorem~\ref{thm:main} therefore applies with the entire REM series present.

Write $N_{i,0}\sim\Pois(\lambda K)$, $N_{i,d}\sim\Pois(d)$,
$M_0\sim\Pois(\lambda(K-1))$, and $M_d\sim\Pois(d-1)$ for
independent counts. The cavity field is
\begin{align}
 A_i(m,\varsigma)
 &=\sum_{a=1}^{N_{i,0}}
       \theta_{i,a}\bigl(m(v_{i,0,a,1}),\ldots,
                         m(v_{i,0,a,K-1}),\varsigma\bigr)
       +g_i\frac{1+\varsigma}{2}\notag\\
 &\quad+\sum_{d\ge2}\sum_{a=1}^{N_{i,d}}
       \theta_{i,d,a}^\epsilon\bigl(m(v_{i,d,a,1}),\ldots,
                                  m(v_{i,d,a,d-1}),\varsigma\bigr),\label{eq:REMfield}
\end{align}
and its correction field is
\begin{align}
 B(m)
 &=\sum_{a=1}^{M_0}
       \theta_a\bigl(m(v_{0,a,1}),\ldots,
                             m(v_{0,a,K})\bigr)\notag\\
 &\quad+\sum_{d\ge2}\sum_{a=1}^{M_d}
       \theta_{d,a}^\epsilon\bigl(m(v_{d,a,1}),\ldots,
                                        m(v_{d,a,d})\bigr).
                                                        \label{eq:REMcorrection}
\end{align}
The clauses and site variables in the correction field are independent
of those in the cavity fields, and all are independent of $G$.
The two regularizations contribute no additional terms to the limiting
fields.

Let $\cG_\epsilon$ be the class of random spin laws satisfying
\eqref{eq:GG} and \eqref{eq:SC} for the field \eqref{eq:REMfield}, and
let $\cP_\epsilon$ be \eqref{eq:functional} with
\eqref{eq:REMfield}--\eqref{eq:REMcorrection}. The notation and
definitions of the fields, the class $\cG_\epsilon$, and the functional
$\cP_\epsilon$ agree with \cite[Section 3]{MPformula}.

\begin{corollary}[Cavity lower bound with the REM perturbation]\label{cor:A4}
For every fixed $\epsilon>0$ there is $G_\epsilon\in\cG_\epsilon$,
obtained as a subsequential spin limit of the regularized models with
deterministically chosen Gaussian coefficients, such that
\begin{equation}\label{eq:A4}
 \liminf_{N\to\infty}F_N^\epsilon
       \ge\cP_\epsilon(G_\epsilon)
       \ge\inf_{G\in\cG_\epsilon}\cP_\epsilon(G).
\end{equation}
Here the regularizations have vanishing free-energy cost per spin, and the limiting identities contain the original clauses and the full extensive REM perturbation.
\end{corollary}

This is the cavity lower bound stated in \cite[Theorem 3.1]{MPformula}.
Its proof uses only the regularization and cavity arguments above; it
does not use the reproducibility theorem or the variational reduction
proved there.

The exponential integrability used in the uniform cavity estimates of
the reproducibility argument also holds for the REM terms. Gaussian
concentration of the maximum gives, for every fixed $t>0$,
\[
 \E e^{tJ_d}\le
    \exp\left(t\E J_d+\tfrac12t^2\epsilon2^{-4d}\right),
 \qquad
 \E(e^{tJ_d}-1)\le C_{t,\epsilon}\sqrt d\,2^{-2d}.
\]
Thus $\sum_{d\ge2}d^q\E(e^{tJ_d}-1)<\infty$ for every fixed
$q\ge0$. Together with the exponential moments of $J_0$ and the
Gaussian unary field, this supplies the required integrability.
The assumption on $J_0$ holds for diluted $K$-spin clauses with
Gaussian or Rademacher couplings and for bounded $K$-sat clauses.

Finally, let $F_N^{\rm model}$ and $\cP_0$ be the free energy per spin
and the cavity functional with the REM perturbation removed. The
absolute-bound comparison gives
\begin{equation}\label{eq:removeREM}
 \sup_N|F_N^\epsilon-F_N^{\rm model}|
       \le\E|g|+\sum_{d\ge2}\E J_d\le C\sqrt\epsilon,
 \qquad
 \sup_G|\cP_\epsilon(G)-\cP_0(G)|\le C\sqrt\epsilon.
\end{equation}
In the second inequality, the supremum is over all random spin laws
$G$, with the same $G$ used in both functionals. The arity weights are
$d$ in the site term and $d-1$ in the correction term, and their sum is
still summable by \eqref{eq:REMmoments}. The constant is independent of
$N$, $G$, and $\epsilon$. These bounds allow the cavity lower bound to
be combined with the structural variational reduction of
\cite{MPformula} at fixed $\epsilon$, followed by $\epsilon\downarrow0$.


\begin{thebibliography}{9}
\bibitem{ASS} M. Aizenman, R. Sims and S. L. Starr. Extended variational principle for the Sherrington-Kirkpatrick spin-glass model. \emph{Phys. Rev. B} \textbf{68} (2003), 214403.
\bibitem{Aldous} D. Aldous. Representations for partially exchangeable arrays of random variables. \emph{J. Multivariate Anal.} \textbf{11} (1981), 581--598.
\bibitem{GG} S. Ghirlanda and F. Guerra. General properties of overlap probability distributions in disordered spin systems. Towards Parisi ultrametricity. \emph{J. Phys. A} \textbf{31} (1998), 9149--9155.
\bibitem{Hoover} D. N. Hoover. Row-column exchangeability and a generalized model for probability. In \emph{Exchangeability in Probability and Statistics}, North-Holland (1982), 281--291.
\bibitem{Spins} D. Panchenko. Spin glass models from the point of view of spin distributions. \emph{Ann. Probab.} \textbf{41} (2013), 1315--1361.
\bibitem{SKbook} D. Panchenko. \emph{The Sherrington--Kirkpatrick Model}. Springer Monographs in Mathematics, Springer, New York (2013).
\bibitem{rRSB} D. Panchenko. Structure of finite-RSB asymptotic Gibbs measures in the diluted spin glass models. \emph{J. Statist. Phys.} \textbf{162} (2016), 1--42.
\bibitem{MPformula} D. Panchenko. \emph{The M\'ezard-Parisi formula}. Preprint (2026).
\end{thebibliography}
\end{document}